\documentclass[preprint,12pt]{elsarticle}

\usepackage{graphicx}
\usepackage{amssymb}
\usepackage{amsmath}
\usepackage{mathtools}
\usepackage{seqsplit}
\usepackage[utf8]{inputenc}
\usepackage{float}
\usepackage{longtable}
\usepackage{url}
\usepackage{enumitem}
\usepackage{indentfirst}

\usepackage{lineno}
\usepackage{amsthm}
\usepackage{dsfont}
\usepackage{hyperref}
\usepackage{tikz}
\usepackage[left=2cm, right=2cm, top=2cm, bottom=2cm]{geometry}
\usetikzlibrary{positioning}

\tikzset{sgplattice/.style={inner sep=1pt,norm/.style={red!50!blue},char/.style={blue!50!black},
		lin/.style={black!50}},cnj/.style={black!50,yshift=-2.5pt,left=-1pt of #1,scale=0.5,fill=white}}
\usepackage{cleveref}

\newcommand{\Mono}{\operatorname{Mono}}  
\newcommand{\Sq}{\operatorname{Sq}}  
\newcommand{\Adj}{\operatorname{Adj}}  
\newcommand{\Op}{\operatorname{Op}} 

\newcommand{\Ver}{\operatorname{Ver}}  
\newcommand{\Edg}{\operatorname{Edg}}  
\newcommand{\Tri}{\operatorname{Tri}}  
\newcommand{\Tet}{\operatorname{Tet}} 

\newcommand{\Sp}{\operatorname{Sp}}

\newcommand{\Sym}{\operatorname{Sym}}
\newcommand{\Free}{\operatorname{Free}}
\newcommand{\ap}{\operatorname{ap}}

\newcommand{\ZZ}{\mathbb{Z}}

\newcommand{\Aut}{\operatorname{Aut}}
\newcommand{\SAut}{\operatorname{SAut}}
\newcommand{\SL}{\operatorname{SL}}
\newcommand{\diag}{\operatorname{diag}}
\newcommand{\tr}{\operatorname{tr}}

\newtheorem{theorem}{Theorem}[section]
\newtheorem{definition}[theorem]{Definition}
\newtheorem{lemma}[theorem]{Lemma}

\newtheorem{proposition}[theorem]{Proposition}
\newtheorem{remark}[theorem]{Remark}

\newtheorem*{remark*}{Remark}

\newtheorem*{theorem*}{Theorem}
\newtheorem*{conjecture*}{Conjecture}
\newtheorem*{question*}{Question}

\numberwithin{table}{section}

\makeatletter
\def\ps@pprintTitle{%
	\let\@oddhead\@empty
	\let\@evenhead\@empty
	\def\@oddfoot{}%
	\let\@evenfoot\@oddfoot
}
\makeatother

\date{}  

\begin{document}
\begin{frontmatter}


\title{Inducing spectral gaps\\
for the cohomological Laplacians of $\Sp_{2n}(\ZZ)$}



\author{Piotr Mizerka}
\author{Jakub Szymański}

\begin{abstract}
We show that the spectral gap of the first cohomological Laplacian $\Delta_1$ for $\Sp_{2n}(\mathbb{Z})$ follows once a slightly stronger assumption holds for some $\Sp_{2m}(\mathbb{Z})$, where $n\geq m$. As an application of this result, we provide explicit lower bounds for some quotients of  $\Sp_{2n}(\mathbb{Z})$ for any $n\geq 2$.
\end{abstract}




\end{frontmatter}



\setcounter{section}{-1}

\section{Introduction}
Recently computer-assisted proofs of property (T) have gained a lot of attention. One of the major outcomes of this approach was the proof of property (T) for $\Aut(F_n)$, the automorphism groups of free groups on $n$ letters for $n\geq 5$ \cite{kkn,kno}, as well as \cite{nitsche2022computerproofspropertyt} for $n=4$. Another advantage of such approach is providing explicit lower bounds for \emph{Kazhdan constants} which are closely related to the pace at which expander families defined by property (T) groups get connected and to the speed of generating random elements in groups \cite{LubotzkyPak,kkn}.

Suppose $G$ is a group generated by a finite symmetric set $S$ and define the Laplacian $\Delta=\frac{1}{2}\sum_{s\in S}s$. It is known that property (T) of $G$ is equivalent to the existence of a positive \emph{spectral gap} $\lambda$, i.e. a positive constant such that $\Delta^2-\lambda\Delta=\sum_i\xi_i^*\xi_i$ for some finite number of elements $\xi_i\in\mathbb{R}G$ \cite{OzawaKlasyk}. If such a positive $\lambda$ exists, it turns out to provide a lower bound for the Kazhdan constant $\kappa$:
$$
\kappa\geq\sqrt{\frac{2\lambda}{|S|}}.
$$
Recently, the work of U. Bader and P.W. Nowak \cite{badernowak} allowed to formulate another condition for property (T), involving explicitly the whole presentation of a given group. Let $G$ be a finitely presented group $G$ generated by a set $\mathcal{S}$. The \emph{Fox calculus} \cite{fox1,fox2}, its cohomological properties \cite{lyndon} and the results from \cite{badernowak} and \cite{bader2023higherkazhdanpropertyunitary} provide us with the Laplacian $\Delta_1\in\mathbb{M}_{|\mathcal{S}|\times|\mathcal{S}|}(\mathbb{R}G)$ (the precise definition of $\Delta_1$ is given in \cref{section:1}) such that the following conditions are equivalent:
\begin{enumerate}[label=(\arabic*)]
    \item There exists a positive $\lambda>0$ such that $\Delta_1-\lambda I_{|\mathcal{S}|}=\sum_i M_i^* M_i$ for some finite number of $M_i\in\mathbb{M}_{|\mathcal{S}|\times|\mathcal{S}|}(\mathbb{R}G)$.
    \item $G$ has property (T).
\end{enumerate}
\begin{remark}
\emph{The implication $(1)\Rightarrow(2)$ follows from the paper of U. Bader and P.W. Nowak \cite{badernowak}, while the converse implication from the recent paper of U. Bader and R. Sauer \cite{bader2023higherkazhdanpropertyunitary} (see Theorem 3.11 and the paragraph below it there). Originally, the assumption $(1)$ was proved in \cite{badernowak} to imply vanishing of the first cohomologies of a given group with coefficients in any unitary representations. The latter statement is known, however, to be equivalent to property (T) for finitely presented groups}.
\end{remark}

In this paper we are focused of $G_n=\Sp_{2n}(\mathbb{Z})$, the group of $2n\times 2n$ symplectic matrices over integers, that is, comprising of integer-valued matrices $A$ satisfying the relation
$$
A^TJ_{2n}A=J_{2n},
$$
where
$$
J_{2n}=\begin{bmatrix}
    0&I_n\\
    -I_n&0
\end{bmatrix}.
$$
These groups admit a presentation defined by Steinberg generators $\mathcal{S}_n$ and relations \cite{spnpres}. We develop an induction technique which allows to conclude the existence of positive $\lambda$ for $G_n$ once we know it for $G_m$ whenever $m\leq n$. This technique involves decomposing $\Delta_1$ into specific summands, with particular focus on the summand $\Adj$ (\cref{section:decomp}). The main theorem is the following.
\begin{theorem}[\Cref{theorem:main}]\label{theorem:mainintro}
Suppose there exists $\lambda>0$ such that $\Adj-\lambda I_{|\mathcal{S}_m|}$ is a sum of squares in $\mathbb{M}_{|\mathcal{S}_m|\times|\mathcal{S}_m|}(\mathbb{R}G_m)$. Then, for any $n\geq m$,
$$
\Delta_1-\frac{n-2}{m-2}\lambda I_{|\mathcal{S}_n|}
$$
is a sum of squares in $\mathbb{M}_{|\mathcal{S}_n|\times|\mathcal{S}_n|}(\mathbb{R}G_n)$, that is there exist matrices $M_1,\ldots,M_k\in\mathbb{M}_{|\mathcal{S}_n|\times|\mathcal{S}_n|}(\mathbb{R}G_n)$ such that $\Delta_1-\frac{n-2}{m-2}\lambda I_{|\mathcal{S}_n|}=M_1^*M_1+\ldots M_k^*M_k$.
\end{theorem}
Our induction technique is inspired by the analogous technique established for $\SL_n(\mathbb{Z})$ and $\SAut(F_n)$, the group of $n\times n$ integer matrices with determinant $1$ and the group of \emph{special atumorphisms} of the free group on $n$ generators respectively, see \cite{mizerka2024inducingspectralgapscohomological}. Both induction approaches are, in turn, inspired by the induction technique for the Ozawa's expression $\Delta^2-\lambda\Delta$, developed in \cite{kkn}. Recently, M. Kaluba and D. Kielak \cite{KalubaKielak2024} developed as well an induction techinque for $\Delta^2-\lambda\Delta$ for Chevalley groups, including $G_n$, allowing one to obtain concrete lower-bound estimates for such $\lambda$.

This paper is organized as follows. In the first section we introduce \emph{Fox derivatives} and define $\Delta_1$. Next, we present the Steinberg presentation of $\Sp_{2n}(\mathbb{Z})$. In \cref{section:decomp} we intruduce the decomposition of $\Delta_1$ into four summmands which behave better with respect to induction. \Cref{section:induction} is the core section of this article -- we show the aforementioned good behaviour of the summands. In \cref{section:gaps}, we apply our induction technique to $\Sp_{2n}(\mathbb{Z})$ which allows us to provide lower bounds for the spectral gap $\lambda$ for specific quotients $H_n$ of $\Sp_{2n}(\mathbb{Z})$. In Appendix, we comment on the computational part of our results: we describe \emph{Wedderburn} approach to accalerate the calculations and improve the \emph{certification} method so that worse accuracy of the numeric solution still can be turned into a rigorous proof.

To simplify the presentation of the material, we introduce the notation $M\geq N$ to denote the fact that the matrix difference $M-N$ is a sum of squares (the ambient space shall be clear from the context).

\section{Fox derivatives and the first cohomological Laplacian}\label{section:1}
 In this subsection, we define the first cohomological Laplacian $\Delta_1$ mentioned in the introduction. This goes back to Fox and Lyndon \cite{fox1,fox2,lyndon}. For any group $G$ and a ring $R$, there is the natural $*$-involution on $RG$, defined by flipping the coefficients of the group elements and their inverses. This $*$-involution induces the operation $*\colon \mathbb{M}_{k\times l}(RG)\rightarrow \mathbb{M}_{l\times k}(RG)$ by precomposing it with the matrix transposition. 
 
 Let $G=\langle s_1,\ldots,s_n|\mathcal{R}\rangle$ be a finitely presented group. Fix a subset $\{r_1,\ldots,r_m\}\subseteq\mathcal{R}$ of its relations. We distungiuish the following matrices in the group ring $\mathbb{Z}G$:
 \begin{gather*}
	d_0=\begin{bmatrix}
	    1-s_1\\
        \vdots\\
        1-s_n
	\end{bmatrix}\in\mathbb{M}_{n\times 1}(\mathbb{Z}G)\quad\text{and}\quad
	d_1=\begin{bmatrix}
	    \frac{\partial r_i}{\partial s_j}
	\end{bmatrix}\in\mathbb{M}_{m\times n}(\mathbb{Z}G).
\end{gather*}
The symbols $\frac{\partial r_i}{\partial s_j}$ denote the \emph{Fox derivatives}, defined as follows. Denote by $F_n$ the free group generated by $s_1,\ldots,s_n$. For each generator $s_j$, we distinguish the homomorphism $\frac{\partial}{\partial s_j}:\mathbb{Z}F_n\rightarrow\mathbb{Z}G$ determined by the following conditions:
\begin{gather*}
	\frac{\partial 1}{\partial s_j}=0,\quad
    \frac{\partial s_i}{\partial s_j}=\begin{cases}
		1 & \text{ if $i=j$,} \\
		0 & \text{if $i\neq j$}
	\end{cases},\quad
	\frac{\partial(uv)}{\partial s_j}=\frac{\partial u}{\partial s_j}+u\frac{\partial v}{\partial s_j}.
\end{gather*}
Then, we define 
\[
\Delta_1^-=d_0d_0^*,\quad\Delta_1^+=d_1^*d_1
\]
and
\[
\Delta_1=\Delta_1^-+\Delta_1^+.
\]
In the case $\{r_1,\ldots,r_m\}=\mathcal{R}$, Lyndon \cite{lyndon} showed that the first cohomology of $G$ with coefficients in a $G$-module $M$ is the kernel of $d_1$ divided by the image of $d_0$ from the sequence below.
$$
M\xrightarrow{d_0}M^n\xrightarrow{d_1}M^m.
$$
Though, for our applications we will have a strict inclusion $\{r_1,\ldots,r_m\}\subsetneq\mathcal{R}$, this still allows us to apply the result of Bader-Nowak \cite[The Main Theorem]{badernowak} due to \cite[Lemma 2.1]{kmn}.

\begin{remark}\label{remark:gap_inheritance}
    \emph{There exists a direct relationship between the spectral gap $\lambda$ from the Ozawa's expression $\Delta^2-\lambda\Delta$ and the one from the expression $\Delta_1-\lambda I$. More precisely, if $\mathcal{S}_n$ is a finite generating set of a finitely presented group $G$ such that $\mathcal{S}_n$ is disjoint from the set of its inverses, $\mathcal{S}_n^{-1}$, then, for any $k>0$, once $\Delta_1^-+k\Delta_1^+-\lambda I\geq 0$, then $\Delta^2-\lambda\Delta\geq 0$ as well. More details were discussed in \cite[Remark 2.2]{mizerka2024inducingspectralgapscohomological}}.
\end{remark}

\section{Presentation of $\Sp_{2n}(\mathbb{Z})$}

In this section, we introduce the presentation of $G_n=\Sp_{2n}(\mathbb{Z})$. Composing this presentation with the projection onto the quotient $H_n$, one obtains a presentation of the latter. The presentation of $G_n$ described below follows from \cite[section 2]{spnpres} and the works of Matsumoto and Behr \cite{Matsumoto,Behr}.

The generator set $\mathcal{S}_n$ of $G_n$ (and also of $H_n$ -- it shall be clear from the context whether we consider the generators as elements of $G_n$ or its quotient $H_n$) consists of the following $2n^2$ elements:
\begin{itemize}
    \item \( X_{i,j} \quad \textrm{for } 1 \leq i \neq j\leq n, \)
    
    \item \( Y_{i,j}, Y'_{i,j} \quad \textrm{for } 1 \leq i < j \leq n, \)
    
    \item \( Z_i, Z_i' \quad \textrm{for } 1 \leq i \leq n. \)
\end{itemize}

Abusing the notation for $i>j$ by letting $Y_{i,j}$ and $Y_{i,j}'$ be $Y_{j,i}$ and $Y_{j,i}'$ respectively, the set of relators $\mathcal{R}_n$ is the following.

\begin{itemize}
\item For $i\in\{1,\ldots,n\}$:
\begin{longtable}{p{0.4\textwidth}
    p{0.4\textwidth}}
    \( (Z_iZ_i'^{-1}Z_i)^4.  \)
    \end{longtable}

    \item For pairwise distinct $i,j \in \{1, …, n\}$:

    \begin{longtable}{p{0.4\textwidth}
    p{0.4\textwidth}}
    \( [X_{i,j},Y_{i,j}] Z_i^{-2},  \) & \( [X_{i,j}, Y_{i,j}'] Z_j'^2, \) \\
    \( [X_{i,j}, Z_j]Y_{i,j}^{-1}Z_i^{-1}\) & \( [X_{i,j}, Z_j]Z_i^{-1}Y_{i,j}^{-1}, \) \\
    \( [X_{i,j}, Z_i']Y_{i,j}'Z_j'^{-1}\) & \( [X_{i,j}, Z_i']Z_j'^{-1}Y_{i,j}', \) \\
    \( [Y_{i,j}, Z_i'] Z_jX_{j,i}^{-1}\) & \([Y_{i,j}, Z_i'] X_{j,i}^{-1}Z_j,  \) \\
    \( [Y_{i,j}', Z_i]Z_j'X_{i,j}\) & \([Y_{i,j}', Z_i]X_{i,j}Z_j'.\)

    \end{longtable}
    
    \item For pairwise distinct subscripts $i,j,k \in \{1,…,n\}$:
 
    \begin{longtable}{p{0.3\textwidth} p{0.3\textwidth} p{0.3\textwidth}}
        \( [X_{i,j}, X_{j,k}]X_{i,k}^{-1}, \) \\
        \( [X_{i,j}, Y_{j,k}]Y_{i,k}^{-1}, \) \\
        \( [X_{i,j}, Y_{i,k}']Y_{j,k}', \) \\
        \( [Y_{i,j}, Y_{j,k}']X_{i,k}^{-1}. \)
    \end{longtable}
    
    \item For all pairs of generators \((A,B)\) with not described commutator, except for \((X_{i,j}, X_{j,i})\), \((Y_{i,j}, Y'_{i,j})\), \((Z_i, Z'_i)\):
    \[
    [A,B].
    \]

\end{itemize}

The correspondence between the above generators and concrete matrix elements from $G_n$ is as follows:

\begin{itemize}
    \item \( X_{i,j} \leftrightarrow I_{2n} + E_{i,j} - E_{j+n, i+n}, \)
    
    \item \( Y_{i,j} \leftrightarrow I_{2n} + E_{i,j+n} + E_{j,i+n}, \quad Y'_{i,j} \leftrightarrow I_{2n} + E_{i+n,j} + E_{j+n,i}, \)
    
    \item \( Z_i \leftrightarrow I_{2n} + E_{i,i+n}, \quad Z_i' \leftrightarrow I_{2n} + E_{i+n,i}, \)
\end{itemize}
where $E_{i,j}$ denotes the $2n\times 2n$ matrix with the only non zero entry at the position $(i,j)$ which equals $1$.

The group $H_n$ is the quotient of $G_n$ by the commutator relators:
$$
[Z_i,Z_i']
$$
for any $i\in\{1,\ldots,n\}$.

\section{Decomposition of the Laplacians}\label{section:decomp}

Inspired by inductive proofs from \cite{kkn} and \cite{mizerka2024inducingspectralgapscohomological}, we introduce the decomposition of the Laplacians of $G_n$ and $H_n$ into parts originating from indices occurring in the coordinates in the case of $\Delta_1^-$ and in the relations defining Jacobian in the case of $\Delta_1^+$.

\[
\Delta_1^- = \Mono_n^- + \Sq_n^- + \Adj_n^- + \Op_n^-,
\]
\[
\Delta_1^+ = \Mono_n^++\Sq_n^+ + \Adj_n^+ + \Op_n^+.
\]
In order to carry out this procedure, let us define the map

\[
\varphi : \Free(\mathcal{S}_n) \longrightarrow \mathcal{P}(\{ 1, 2, …, n\})
\]
from the free group generated by $\mathcal{S}_n$ to subsets of the set $\{1,\ldots,n\}$ which assigns to a word the set of all indices of generators occurring in the reduced form of this word. \\

What is more, we wish to recognise the following subsets of $\mathcal{P}(\{ 1, 2, …, n\})$:
\begin{align*}
    \Ver_n = \left\{ A \in \mathcal{P}(\{ 1, 2, …, n\}) \mid \#A = 1\right\}, \\
    \Edg_n = \left\{ A \in \mathcal{P}(\{ 1, 2, …, n\}) \mid \#A = 2\right\}, \\
    \Tri_n = \left\{ A \in \mathcal{P}(\{ 1, 2, …, n\}) \mid \#A = 3\right\}, \\
    \Tet_n = \left\{ A \in \mathcal{P}(\{ 1, 2, …, n\}) \mid \#A = 4\right\}.
\end{align*}
Consequently, the decomposition of the negative part of the Laplacian is obtained as follows.
\begin{align*}
    \left( \Mono_n^- \right)_{s,t} = 
    \left( \Delta_1^- \right)_{s,t} \text{ if } \varphi(s) \cup \varphi(t) \in \Ver_n \text{ and } 0 \text{ otherwise,} \\
    \left( \Sq_n^- \right)_{s,t} = \left( \Delta_1^- \right)_{s,t} \text{ if } \varphi(s) \cup \varphi(t) \in \Edg_n \text{ and } 0 \text{ otherwise,} \\
    \left( \Adj_n^- \right)_{s,t} = \left( \Delta_1^- \right)_{s,t} \text{ if } \varphi(s) \cup \varphi(t) \in \Tri_n \text{ and } 0 \text{ otherwise,} \\
    \left( \Op_n^- \right)_{s,t} = \left( \Delta_1^- \right)_{s,t} \text{ if } \varphi(s) \cup \varphi(t) \in \Tet_n \text{ and } 0 \text{ otherwise,}
\end{align*}
and the decompoisition of the positive part is given by 

\begin{align*}
\left( \Mono_n^+ \right)_{s,t} = \sum_{\varphi(r) \in \Ver_n} \left( \frac{\partial r}{\partial s} \right)^* \frac{\partial r}{\partial t}, \\
    \left( \Sq_n^+ \right)_{s,t} = \sum_{\varphi(r) \in \Edg_n} \left( \frac{\partial r}{\partial s} \right)^* \frac{\partial r}{\partial t}, \\
    \left( \Adj_n^+ \right)_{s,t} = \sum_{\varphi(r) \in \Tri_n} \left( \frac{\partial r}{\partial s} \right)^* \frac{\partial r}{\partial t}, \\
    \left( \Op_n^+ \right)_{s,t} = \sum_{\varphi(r) \in \Tet_n} \left( \frac{\partial r}{\partial s} \right)^* \frac{\partial r}{\partial t},
\end{align*}
where $r$ is taken from the relators $\mathcal{R}_n$.

\section{Main induction argument}\label{section:induction}

In the following section we examine collectively the symmetrization properties of the matrices $\Mono_n^\pm, \Sq_n^\pm, \Adj_n^\pm, \Op_n^\pm$ which appear in the decompositions of $\Delta_1^-$ and $\Delta_1^+$.

We introduce an action of the symmetric group $\Sym_n$ on the set of generators $\mathcal{S}_n$ and the free group $\Free(\mathcal{S}_n)$ generated by them, by permuting the indices of the generators (e.g. $\sigma X_{i,j}=X_{\sigma(i),\sigma(j)}$, $\sigma Z_i=Z_{\sigma(i)}$). Furthermore, this action can be extended to the case of $\mathbb{M}_{2n^2 \times 2n^2}(\mathbb{R}G_n)$ by the following formula
\[
(\sigma A)_{s,t} = \sigma\left( A_{\sigma^{-1}s, \sigma^{-1}t} \right),
\]
where the action on $\mathbb{R}G_n$ is naturally given by the action on the generators.

\subsection{Expressing the action in the simplex setting}

Suppose $A^m$ is a matrix from the set $\{ \Mono^\pm_m, \Sq^\pm_m, \Adj^\pm_m, \Op^\pm_m \}$ and let $C_{A^m}$ be the set of the simplex faces defining $A^m$ (respectively, for the matrices inside this set it is: $\Ver_m, \Edg_m, \Tri_m, \Tet_m$).

Let us first assume that $A^m$ originates from $\Delta_1^-$. For $\theta \in C_{A^m}$ we define

\[
(A^m_{\theta})_{s,t} =
\begin{cases} 
(1 - s)(1 - t)^* \quad \text{ if } \varphi(s) \cup \varphi(t) = \theta, \\
0 \quad \text{ otherwise.}
\end{cases}
\]
For $\sigma \in \Sym_m$ the image of $A^m_\theta$ under the action of $\sigma$ equals $A^m_{\sigma(\theta)}$ which follows directly from the formula for symmetric group action. 

Next, for $A^m$ derived from $\Delta_1^+$ we define

\[
(A^m_{\theta})_{s,t} = \sum_{\varphi(r) = \theta} \left( \frac{\partial r}{\partial s} \right)^* \frac{\partial r}{\partial t}.
\]
Then, due to \cite[Lemma 5.2]{mizerka2024inducingspectralgapscohomological} and by the invariance of the relator set $\mathcal{R}_m$ under the action of $\Sym_m$, we have

\begin{align*}
    \left(\sigma(A^m_{\theta}) \right) _{s,t} &= \sigma \left( \sum_{\phi(r) = \theta} \left( \frac{\partial r}{\partial \sigma^{-1}s} \right)^* \frac{\partial r}{\partial \sigma^{-1}t} \right)  \\
    &= \sum_{\phi(r) = \sigma(\theta)} \left( \frac{\partial r}{\partial s} \right)^* \frac{\partial r}{\partial t} = ( A^m_{\sigma(\theta)} )_{s,t}.
\end{align*}
Hence, for any $A^m \in \{ \Mono^\pm_m, \Sq^\pm_m, \Adj^\pm_m, \Op^\pm_m \}$ we obtain the following formula:

\[
\sigma(A^m_\theta) = A^m_{\sigma(\theta)}.
\]
Therefore

\[
\sigma(A^m) = \sum_{\theta \in C_{A^m}} \sigma(A^m_\theta) = \sum_{\theta \in C_{A^m}} A^m_{\sigma(\theta)} = A^m,
\]
where the last equality holds because $C_{A^m}$ is invariant under permutations. Therefore, the matrices $\Mono^\pm_m, \Sq^\pm_m, \Adj^\pm_m, \Op^\pm_m  \in \mathbb{M}_{2m^2 \times 2m^2}(\mathbb{R}G_m)$ are invariant under the action of $\Sym_m$.

In what follows, we shall abuse the notations $\sigma(A^m)$ and $\sigma(A_{\theta}^m)$ in the case $\sigma\in\Sym_n$ for some $n\geq m$ and $\theta\in C_{A^m}$ to denote the result of the action of $\sigma$ on the canonical embeddings of $A^m,A_{\theta}^m\in\mathbb{M}_{2m^2\times 2m^2}(\mathbb{R}G_m)$ into $\mathbb{M}_{2n^2\times 2n^2}(\mathbb{R}G_n)$.

\subsection{Symmetrization of the Laplacian}

Let $A^m$ be again a matrix from the set $\{ \Mono^\pm_m, \Sq^\pm_m, \Adj^\pm_m, \Op^\pm_m \}$ and let $k_A$ be the power of the sets defining the simplex faces from $C_{A^m}$ (for the matrices from the mentioned set it is $1,2,3$, and $4$ respectively).

Recall that we have the following decomposition

\[
A^m = \sum_{\theta \in C_{A^m}} A^m_\theta.
\]
By considering a transformation of the above formula we will be able to describe symmetrization properties of the appropriate parts of the Laplacian.

Due to the orbit-stabilizer theorem applied to the action of $\Sym_m$ on $C_{A^m}$, we get

\[
A^m = \sum _{\sigma \in \Sym_m} \frac{1}{(m - k_A)!k_A!} A^m_{\sigma(\{1,…,k_A\})} =  \frac{1}{(m - k_A)!k_A!} \sum_{\sigma \in \Sym_m} \sigma(A^{k_A}).
\]
Suppose that we have a natural number $n \geq m$ and let $\left\{ \tau_i \right\}_{i = 1, …, \frac{n!}{m!}}$ be the set of representatives of the cosets $\Sym_n / \Sym_m$. Then we obtain the following.

\begin{lemma}\label{lemma:nonid_sym}
One has
$$
A^n=\frac{(m - k_A)!}{(n - k_A)! \cdot m!} \sum_{\tau \in \Sym_n} \tau\left( A^{m} \right).
$$
\end{lemma}
\begin{proof}
A direct computation, based on the previous observations:
\begin{align*}
    A^n &= \frac{1}{(n - k_A)!k_A!} \sum_{\sigma \in \Sym_n} \sigma\left(A^{k_A}\right) \\
    &= \frac{1}{(n - k_A)!k_A!} \sum_{i = 1, …, \frac{n!}{m!}} \tau_i \left( \sum_{\tau \in \Sym_m} \tau\left( A^{k_A} \right) \right) \\
    &= \frac{(m - k_A)!}{(n - k_A)!} \sum_{i = 1, …, \frac{n!}{m!}} \tau_i \left( A^m \right) \\
    &= \frac{(m - k_A)!}{(n - k_A)!} \sum_{i = 1, …, \frac{n!}{m!}} \tau_i \left( \frac{1}{m!} \sum_{\tau \in \Sym_m} \tau\left( A^{m} \right) \right) \\
    &= \frac{(m - k_A)!}{(n - k_A)! \cdot m!} \sum_{\tau \in \Sym_n} \tau\left( A^{m} \right).
\end{align*}
\end{proof}

\subsection{Symmetrization of identity matrix}

Let $I_n$ be the identity matrix inside $\mathbb{M}_{2n^2\times 2n^2}(\mathbb{R}G_n)$, that is $I_n$ is equal to $I_{|\mathcal{S}_n|}$ from the Introduction. Let us consider the following split of $I_n$ into two parts: $I_n^{\Mono}$ and $I_n^{\Sq}$ in which diagonal values are non-zero if and only if they are indexed by a generator mapping via $\varphi$ to $\Ver_n$ and $\Edg_n$ respectively. This gives the decomposition
\[
I_n = I_n^{\Mono} + I_n^{\Sq},
\]
with $2n$ nontrivial elements inside $I_n^{\Mono}$ and $2n(n-1)$ inside $I_n^{\Sq}$.

Further in this notes $I_m$, will also denote the canonical embedding of the identity matrix from $\mathbb{M}_{2m^2 \times 2m^2}(\mathbb{R}G_m)$ into $\mathbb{M}_{2n^2 \times 2n^2}(\mathbb{R}G_n)$ for an integer $n \geq m$. \\

We show the following.
\begin{lemma}\label{lemma:id_sym}
\[
\sum_{\sigma \in \Sym_n} \sigma I_m = m(m-1)(n-2)! \cdot I_n^{\Sq} + m(n-1)! \cdot I_n^{\Mono}.
\]
\end{lemma}

\begin{proof}
    Both $\sum_{\sigma \in \Sym_n} \sigma I_m^{\Mono}$ and $\sum_{\sigma \in \Sym_n} \sigma I_m^{\Sq}$ are multiplicities of $I_n^{\Mono}$ and $I_n^{\Sq}$ respectively. If we consider the power of $\Sym_n$ acting on these matrices and the numbers of non-zero entries in each of them we obtain the multiplication factors as in the above formula.
\end{proof}

\subsection{The main theorem}
Below we present the main theorem of this paper. This is the extended (and slightly stronger) version of the theorem from the introduction.
\begin{theorem}[\Cref{theorem:mainintro}]\label{theorem:main}
    Suppose $\Adj_m-\lambda I_m$ is a sum of squares in $\mathbb{M}_{2m^2\times 2m^2}(\mathbb{R}G_m)$ for some $m\geq 2$. Then $\Adj_n-\frac{n-2}{m-2}\lambda I_n$ is a sum of squares in $\mathbb{M}_{2n^2\times2n^2}(\mathbb{R}G_n)$ for any $n\geq m$. In the case only $\Adj_m^-+\Delta_1^+-\lambda I_m$ is a sum of squares in $\mathbb{M}_{2m^2\times 2m^2}(\mathbb{R}G_m)$ for some $m\geq 2$, a slightly weaker conclusion holds: $\Adj_n^-+k\Delta_1^+-\frac{n-2}{m-2}\lambda I_n$ is a sum of squares in $\mathbb{M}_{2n^2\times2n^2}(\mathbb{R}G_n)$ for any $n\geq m$ and $k$ sufficiently large.
\end{theorem}
\begin{proof}
Suppose $\Adj_m-\lambda I_m$ is a sum of squares in $\mathbb{M}_{2m^2\times 2m^2}(\mathbb{R}G_m)$ for some $m\geq 2$. By \Cref{lemma:nonid_sym} we have 
\[
\Adj_n=\frac{(m-3)!}{(n-2)!\cdot m!}\sum_{\sigma\in\Sym_n}\sigma(\Adj_m).
\]
Thus, by \Cref{lemma:id_sym}, 
\begin{align*}
\Adj_n-\frac{n-2}{m-2}\cdot \lambda I_n&=\frac{(m-3)!}{(n-3)!\cdot m!}\sum_{\sigma\in\Sym_n}\sigma(\Adj_m-\lambda I_m)+\frac{n-m}{m-1}\cdot I_n^{\Mono}\\
&\geq \frac{(m-3)!}{(n-3)!\cdot m!}\sum_{\sigma\in\Sym_n}\sigma(\Adj_m-\lambda I_m)\geq 0.  
\end{align*}
The second part can be proved similarly with an additional asymptotic argument. Analogously as is in the first part:
\begin{align}\label{align:adjn}
\Adj_n^--\frac{n-2}{m-2}\cdot \lambda I_n\geq \frac{(m-3)!}{(n-3)!\cdot m!}\sum_{\sigma\in\Sym_n}\sigma(\Adj_m^--\lambda I_m).
\end{align}
Now, by \Cref{lemma:nonid_sym} and \Cref{lemma:id_sym}, we obtain
\begin{align*}
\frac{(n-1)(n-2)}{(m-2)(m-1)}\Delta_n^+&=\frac{(n-1)(n-2)}{(m-2)(m-1)}\left(\Mono_n^++\Sq_n^++\Adj_n^++\Op_n^+\right)\\
&\geq\frac{(m-3)!}{(n-3)!\cdot m!}\sum_{\sigma\in\Sym_n}\sigma(\Delta_m^+).  
\end{align*}
The last inequality in connection with inequality (\ref{align:adjn}) yields the assertion for $k=(n-1)(n-2)/(m-2)(m-1)$.
\end{proof}
\section{Application to spectral gaps related to $\Sp_{2n}(\mathbb{Z})$}\label{section:gaps}
In this section we show, using computer calulations, lower bounds for positive $\lambda$ such that $\Delta_1-\lambda I$ is a sum of squares for $G_n=\Sp_{2n}(\mathbb{Z})$ for $n\in\{2,3\}$ and $H_n=\Sp_{2n}(\mathbb{Z})/\langle [Z_i,Z_i']\rangle$ for any $n\geq 2$.
\begin{lemma}\label{mono_op_give_sos}
    The matrices $\Mono^-_n$ and $\Op^+_n + \Op^-_n$ are sums of squares.
\end{lemma}

\begin{proof}
    For each index $i \in\{1, 2, … n \}$  let us define $d^{i}$ as the following column vector inside $\mathbb{M}_{2n^2 \times 1}(\mathbb{R}G_n)$
    \[
    (d^{i})_s = \begin{cases}
        1 - s \quad \text{if } \varphi(s)=\{i\}, \\
        0  \quad \text{otherwise}.
    \end{cases}
    \]
    Then, we can present $\Mono^-_n$ as the following sum of squares

    \[
    \Mono^-_n = \sum_{i=1}^n d^{i} (d^{i})^*.
    \]
    Next, for any pair of generators $s,t$ such that $\varphi(s) \cup \varphi(t) \in \Tet_n$, we have:
    \begin{align*}
    (\Op^+_n)_{s,t} &= \left(\frac{\partial [s,t]}{\partial s}\right)^*\frac{\partial [s,t]}{\partial t} + \left(\frac{\partial [t,s]}{\partial s}\right)^*\frac{\partial [t,s]}{\partial t} \\
    &= -2(1-s)(1-t)^* = -2 (\Op_n^-)_{s,t}
    \end{align*}
    and for other generators $s\neq t$, we get $(\Op^+_n)_{s,t} = 0$. Moreover, since the elements on the diagonal of $\Op_n^-$ equal zero and the diagonal part of $\Op_n^+$, denoted by $\diag(\Op_n^+)$ is a sum of squares, we obtain
    \[
    \Op_n^+ + \Op_n^- = \frac{1}{2}\diag(\Op_n^+) + \frac{1}{2}\Op_n^+ \geq 0.
    \]

\end{proof}

A direct computation on a computer showed the following.
\begin{lemma}\label{lemma:computations_H3}
    $\Adj^-+\Delta_1^+-0.24I_3$ and $\Sq^-+\Delta_1^+-0.99 I_3$ are a sums of squares in $\mathbb{M}_{18\times18}(\mathbb{R}H_3)$.
\end{lemma}
\begin{remark}
    \emph{Due to a a usage of a direct implementation of symplectic groups from \emph{Groups.jl} package of M. Kaluba \cite{kno,kkn}, the computations in the lemma above were done over a subspace of the group ring $\mathbb{R}G_n$ instead of $\mathbb{R}H_n$. This implies the desired statements for $H_n$ as it is a quotient of $G_n$.}
\end{remark}
\Cref{lemma:computations_H3} leads us, as a corollary, to the following.

\begin{theorem}
The element $\Delta_1^-+k\Delta_1^+-\left((n-2)0.24+0.99\right)I_n$ is a sum of squares in $\mathbb{M}_{2n^2\times 2n^2}(\mathbb{R}H_n)$ for any $n\geq 2$ and $k$ sufficiently large.    
\end{theorem}
\begin{proof}
It follows from \Cref{theorem:main} and \Cref{lemma:computations_H3} that there exists $k'\geq 0$ such that 
$$
\Adj_n^-+k\Delta_1^+-0.24(n-2)I_n\geq 0.
$$
Similarly, by \Cref{lemma:nonid_sym,lemma:id_sym} and \Cref{lemma:computations_H3}, we get 
$$
\Sq_n^-+\Delta_1^+-0.99I_n\geq 0.
$$
On the other hand, we have from \Cref{mono_op_give_sos} that $\Mono_n^-$ and $\Op_n^-+\Op_n^+$ are sums of squares. Summing the considered four non-negative expressions together, we conclude that the expression $\Delta_1^-+k\Delta_1^+-\left((n-2)0.24+0.99\right)I_n$ is a sum of squares in $\mathbb{M}_{2n^2\times 2n^2}(\mathbb{R}H_n)$ for any $n\geq 2$ and $k$ big enough.
\end{proof}

\begin{remark}
    \emph{The induction argument can also be performed without knowing that $\Sq^-+\Delta_1^+-0.99 I_3$ are a sums of squares in $\mathbb{M}_{18\times18}(\mathbb{R}H_3)$. In fact, since $\Sq_2^-$ is $*$-invariant, there exists a constant $\lambda_0$ such that $\Sq_2^- + \lambda_0 I_{2}$ (cf. \Cref{identity_is_an_order_unit})  is a sum of squares. Therefore, by \Cref{lemma:nonid_sym,lemma:id_sym} (we can apply the same argument as in \Cref{lemma:nonid_sym} to show that $\Sq_n^-$ and $\lambda_0 I_n^{\Sq}$ symmetrize at the same pace), we get that for any $n \geq 2$ the following expression also admits a sum of squares decomposition:
    \[
    \Sq_n^- + \lambda_0 I_n^{\Sq} + (n-1)\lambda_0 I_n^{\Mono}.
    \]
    Combining this with the symmetrized version of our base calculation for the $\Adj$ part and the non-negative summands from \Cref{mono_op_give_sos}, we obtain for $k$ sufficiently large:
    \begin{align*}
    0&\leq\left(\Adj_n^-+k\Delta_1^+-0.24(n-2)I_n^{\Sq}-0.24(n-1)(n-2)I_n^{\Mono}\right)+\left(\Sq_n^- + \lambda I_n^{\Sq} + (n-1)\lambda_0 I_n^{\Mono}\right)\\
    &\leq \Delta_1^-+k\Delta_1^+-0.24(n-2)I_n.
    \end{align*}
    The last inequality holds for $n$ big enough and requires adding $\Mono_n^-$ and $\Op_n^-+\Op_n^+$ which are non-negative.}
\end{remark}

On the other hand, applying another computer calculations, we were able to provide the lower bounds for $G_2$ and $G_3$.
\begin{theorem}\label{theorem:bounds_spn}
    $\Delta_1-0.82 I_2$ is a sum of squares in $\mathbb{M}_{8\times 8}(\mathbb{R}G_2)$ and $\Delta_1-0.99 I_3$ is a sum of squares in $\mathbb{M}_{18\times 18}(\mathbb{R}G_3)$.
\end{theorem}
\begin{remark}
    \emph{It follows from \Cref{theorem:bounds_spn} and \Cref{remark:gap_inheritance} that the lower bounds for $\lambda$ in the expression $\Delta^2-\lambda\Delta$ from Ozawa's paper \cite{OzawaKlasyk} are $0.82$ and $0.99$ for $G_2$ and $G_3$ respectively. This still, however, is no better than the estimates provided by M. Kaluba and D. Kielak which are $1.41$ and $2.41$ respectively, see Theorem 3.12 and Theorem 3.16 from \cite{KalubaKielak2024}.}
\end{remark}
\subsection{Replication details}
We implemented the necessary code to define and compute the sums of squares decompositions of interest in Julia \cite{bezanson2017julia} programming language. The replication instructions can be found on the GitHub page of our repository \cite{MizerkaSzymanski}. 

The justification for transforming the numerical sum of squares approximation into the exact version is discussed in section 3.2 of \cite{kmn}. In this paper we improve this transformation method in Appendix B. Additional information on obtaining the numerical approximation is also available in \cite{kmn}. In the current article, we also apply an acceleration to finding such a numerical approximation by means of \emph{Wedderburn decomposition}. This is an application of the approach from \cite[section 2]{kno} to the matrix case over group rings. We discuss this approach in the appendix below. 

For the computations concerning the Wedderburn decomposition we have used \emph{SymbolicWedderburn.jl} package of M. Kaluba \cite{kno,kkn}. The necessary implementations are contained in "wedderburn.jl" in the "src" folder of our package \cite{MizerkaSzymanski}. Other dependencies used in our package are listed in "Project.toml" file from \cite{MizerkaSzymanski}. In particular, to ensure the rigor of our computations, we have applied \emph{IntervalArithmetic.jl} package \cite{IntervalArithmeticjl}, and for the numerical computations involved in the semi-positive definite optimization, we have used \emph{JuMP} \cite{Lubin2023} and \emph{SCS} solver \cite{ocpb:16}.
\newline
\appendix
\noindent\textbf{\noindent Appendix: Methods for faster calculation of spectral gaps and certification of the solution}

In this appendix, we present two methods for faster computation related to finding lower bounds for spectral gaps of cohomological Laplacians. The contents of the first part of the appendix was inspired by the unpublished notes of P.W. Nowak and the first author and the contents of the second part partially overlaps with the master thesis of the second author entitled "Spectral gaps for the Property (T) groups".

\section{Wedderburn decomposition for the matrix case}

Let $G=\langle s_1, ..., s_n | r_1, ..., r_k \rangle$. As shown in \cite{kmn}, given a matrix $M=M^*=[m_{i,j}]\in\mathbb{M}_{n}(\mathbb{R}G)$, the expression $M - \lambda I_n$ is a sum of squares for a constant $\lambda \in \mathbb{R}$ if and only if there exists a finite subset $E=\{g_1,\ldots,g_m\} \subseteq G$ for which the following set is non-empty:
    \begin{align*}
        F(M,\lambda) = \{ P \in \mathbb{M}_{nm}(\mathbb{R}) &\cong \mathbb{M}_{n}(\mathbb{R}) \otimes \mathbb{M}_{E}(\mathbb{R})\colon \quad \\
         &P \text{ is positive semidefinite  and }& \\
         & \forall_{i,j \in [n], g \in G} \hspace{0.5cm} m_{i,j}(g) - \delta_{i,j} \lambda = \langle E_{i,j} \otimes A_g, P \rangle \} \neq \emptyset, 
    \end{align*}
where $[n]=\{1,\ldots,n\}$, $\langle A,B\rangle=\tr(A^*B)$, and $A_g \in \mathbb{M}_E(\mathbb{R})$ is given for any $g \in G$ as follows 

\[
(A_g)_{g_i,g_j} = 
    \begin{cases}
        1 & \text{for } g_i^{*} g_j = g, \\
        0 & \text{otherwise.}
    \end{cases}
\]
The aim of this section is to generalize section 2 of \cite{kno} and use similar block-diagonalization in the case of our optimization problem for the first cohomological laplacian. 

We assume that a symmetric group $\Sigma$ acts on $G$ and that each element $\sigma \in \Sigma$ acts as a permutation on the set of generators, the set of relations, and the subset $E$.

\begin{definition}
    Let $X$ be either $\mathcal{S}=\{s_1,\ldots,s_n\}$ or $E$. One can define the following induced linear actions:
    \begin{align*}
        \Sigma \curvearrowright \mathbb{R}G &\colon \quad \sigma  \left( \sum_{g \in G} \lambda_g g \right) = \sum_{g \in G} \lambda_{\sigma ^{-1}(g)}g, \\
        \Sigma \curvearrowright \mathbb{M}_{X}(\mathbb{R}) &\colon \quad (\sigma T)_{g_i,g_j} = T_{\sigma^{-1}g_i,\sigma^{-1}g_j}.
    \end{align*}
\end{definition}
\noindent Based on the aforementioned actions we introduce also linear actions of $\Sigma$ on $\mathbb{M}_S(\mathbb{R}G) \cong \mathbb{M}_S(\mathbb{R}) \otimes \mathbb{R}G$ and $\mathbb{M}_{S \times E}(\mathbb{R}) \cong \mathbb{M}_S (\mathbb{R}) \otimes \mathbb{M}_E(\mathbb{R})$ by acting component-wise on each factor of the tensor products. 

The following result is a slight improvement of Corollary 3.5 from the work of \cite{kmn} if we take into account and use the $\Sigma$-action.

\begin{proposition}

    Let $M \in \mathbb{M}_{S}{(\mathbb{R}G)}$ be a $\Sigma$-invariant matrix which can be decomposed into a sum of squares. Then there exists a $\Sigma$-invariant positive semidefinite matrix $P \in \mathbb{M}_{S \times E}(\mathbb{R})$ such that the following equality holds
    \[
    M = x_E^* P x_E,
    \]
    where $x_E$ is defined by the following tensor product:
    \[
    x_E = I_n \otimes (g_1,…,g_m)^T \in \mathbb{M}_{mn \times n}(\mathbb{R}G).
    \]
\end{proposition}

\begin{definition} \label{definition:sigma-repr}
    One can define the real representation $\varrho$ of $\Sigma$ on $\mathbb{R}(S \times E)$ as follows:
    \[
    \varrho(\sigma) \in \mathbb{M}_{S \times E}(\mathbb{R}), \quad 
    \varrho(\sigma)_{(i,g),(j,h)} =     \begin{cases}
        1 & \text{for } (i,g) = (\sigma(j), \sigma h), \\
        0 & \text{otherwise.}
    \end{cases}
    \]
    Then, let $A(\varrho)$ denote the real algebra generated by the elements $\{ \varrho(\sigma) \}_{\sigma \in \Sigma}$.
\end{definition}
According to Maschke's theorem, the representation $\varrho$ is semisimple. By the definition of the associated algebra $A(\varrho)$, it is isomorphic to the direct sum of real algebras corresponding to real irreducible representations of $\Sigma$:
\[
A(\varrho) \cong \bigoplus_{\pi \in \operatorname{Irr}^{\mathbb{R}}(\Sigma)} I_{m_\pi} \otimes A(\pi) \subseteq \mathbb{M}_{S \times E}(\mathbb{R}).
\]
Since the only endomorphisms of such representations are scalar multiplications (\cite[Chapter 4]{James}), also the following isomorphism of the commutant of $A(\varrho)$ holds:
\[
A(\varrho)'\cong \bigoplus_{\pi \in \operatorname{Irr}^{\mathbb{R}}(\Sigma)} \mathbb{M}_{m_\pi}(\mathbb{R}) \otimes I_{\operatorname{dim}(\pi)}.
\]
    
\begin{proposition}
    The set of $\Sigma$-invariant matrices $\mathbb{M}^\Sigma_{S \times E}(\mathbb{R})$ coincides with $A(\varrho)'$.  
\end{proposition}

\begin{proof}
    First, let us consider the matrix $M = \sum_{i,j \in [n]} E_{i,j} \otimes M^{i,j} \in \mathbb{M}_{S \times E}^{\Sigma}(\mathbb{R})$. The invariance of M means that the following equalities hold for each $\sigma \in \Sigma$

    \begin{align*}
        \sum_{i,j \in [n]} E_{i,j} \otimes M^{i,j} &= \sigma\left(\sum_{i,j \in [n]} E_{i,j} \otimes M^{i,j}\right) \\
        &= \sum_{i,j \in [n]} E_{\sigma(i),\sigma(j)} \otimes \sigma M^{i,j} \\
        &= \sum_{i,j \in [n]} E_{i,j} \otimes \sigma M^{\sigma^{-1}(i),\sigma^{-1}(j)}.
    \end{align*}
    Therefore, we have the equality of the matrices $\sigma M^{\sigma^{-1}(i),\sigma^{-1}(j)}$ and $M^{i,j}$ for all $i,j \in [n]$, which implies that on each coordinate $(g,h) \in E^2$ we obtain

    \[
    M^{i,j}_{g,h} = \left( \sigma M^{\sigma^{-1}(i),\sigma^{-1}(j)} \right)_{g,h} = M^{\sigma^{-1}(i),\sigma^{-1}(j)}_{\sigma^{-1}(g),\sigma^{-1}(h)}. \tag{$\star$}
    \]
    On the other side, we may apply the same convention for block description in the case of matrices given by $\varrho$ and obtain the equality

    \[
    \varrho(\sigma) = \sum_{i \in [n], g \in G} E_{i, \sigma^{-1}(i)} \otimes E_{g, \sigma^{-1}(g)}.
    \]
    Thus, for a matrix $M = \sum_{i,j \in [n]} E_{i,j} \otimes M^{i,j} \in \mathbb{M}_{S \times E}(\mathbb{R})$ we can investigate the action of $\varrho$ from the both sides:
    \begin{align*}
        \varrho(\sigma) M &= \left(\sum_{i \in [n], g \in E} E_{i, \sigma^{-1}(i)} \otimes E_{g, \sigma^{-1}(g)}\right) \left(\sum_{i,j \in [n]} E_{i,j} \otimes M^{i,j} \right) \\
        &= \sum_{i,j \in [n], g \in E} E_{i,j} \otimes \left( E_{g, \sigma^{-1}(g)} \cdot M^{\sigma^{-1}(i),j} \right) \\
        M \varrho(\sigma) &= \sum_{i,j \in [n], h \in E}  E_{i,j} \otimes \left( M^{i, \sigma(j)} \cdot E_{\sigma(h), h} \right).
    \end{align*}
    The matrix $M$ is contained in $A(\varrho)'$ if and only if the above matrices are equal for each $\sigma \in \Sigma$. Based on the above calculations, it follows that the entry of $\varrho(\sigma)M = M \varrho(\sigma)$ with coordinates $\left((i,g),(j,h)\right)$ is equal to

    \[
    M^{\sigma^{-1}(i),j}_{\sigma^{-1}(g), h} =  M^{i,\sigma(j)}_{g,\sigma(h)}
    \]
    which is the same equation as ($\star$) up to the renaming of the variables, and so the stated equality holds.
\end{proof}

By the previous proposition, we have the isomorphism of $\Sigma$-invariant matrices and the algebra $\bigoplus_{\pi \in \operatorname{Irr}^{\mathbb{R}}(\Sigma)} \mathbb{M}_{m_\pi}(\mathbb{R}) \otimes I_{\operatorname{dim}(\pi)}$. 
Suppose that it is given by the following map: 
\[
\Theta\colon \mathbb{M}_{S \times E}^{\Sigma}(\mathbb{R}) \rightarrow \bigoplus_{\pi \in \operatorname{Irr}^{\mathbb{R}}(\Sigma)} \mathbb{M}_{m_\pi}(\mathbb{R}), \quad \Theta(M) = \bigoplus_{\pi \in \operatorname{Irr}^{\mathbb{R}}(\Sigma)} U_\pi M U_\pi^T.
\]

Then, basing on the matrices $U _\pi$, we can define the block semidefinite problem $\widetilde{F}(M,\lambda)$ and by the analogue of \cite[ Proposition 9]{kno} its solution will always imply the solution of the problem $F(M,\lambda)$:
\begin{align*}
    \widetilde{F}(M,\lambda) = \{ & \{P_\pi \in \mathbb{M}_{m_\pi}(\mathbb{R}) \}_{\pi \in \operatorname{Irr}^{\mathbb{R}}(\Sigma)}: \quad \\
     &\forall_{\pi} P_\pi \text{ are positive semidefinite and }  \\
     & m_{i,j}(g) - \delta_{i,j} \lambda = \sum_{\pi} \langle U_\pi A_{B} U_\pi^T, P_\pi \rangle \text{ for any orbit } B = [E_{i,j} \otimes A_g] \in \mathcal{B}/\Sigma \},  
\end{align*}
    where $\mathcal{B} = \left\{ E_{i,j} \otimes A_g: i,j \in [n] \text{ and } g \in E^*E \right\}$ and matrices $A_B$ are given as follows
    
    \[
    A_{B} = \frac{1}{|\Sigma|} \sum_{X \in B} X.
    \]





\section{Improved certification of the solution}
Let $G$ be a finitely presented group and $\Delta_1$ its first cohomological Laplacian. Recall (e.g. from section 3.2 of \cite{kmn}) that in order to turn a numerical approximation $(X_{\ap},\lambda_{\ap})$ of the Laplacian, $\Delta_1-\lambda_{\ap} I\approx X_{\ap}$, into a rigorous proof of the existence of a positive $\lambda$ such that $\Delta_1-\lambda I\geq 0$, we have to play with the residual error $r=\Delta_1-\lambda_{\ap} I-X_{\ap}$ (the operations are performed in \emph{interval arithmetic}). We can ensure that $X_{\ap}$ is a sum of squares which allows us to add a specific expression $r+\alpha_r I\geq 0$ to it. The crucial part is to have a good control over $\alpha_r$. This control estimate was introduced in \cite[Proposition 3.2]{kmn} and we improve it here.

\begin{definition}\label{definition:order_unit}
    Let $U = (u_1, …, u_n)$ be a vector with coefficients in a group ring $\mathbb{R}G$. Then the first norm of $U$ is given by
    \[
    ||U||_1 \coloneq \sum_{i=1}^n ||u_i||_1,
    \]
    where $||\sum_g\lambda_gg||_1=\sum_g|\lambda_g|$.
\end{definition}

\begin{theorem}\label{identity_is_an_order_unit}
    For any matrix $M=M^*\in\mathbb{M}_{n\times n}(\mathbb{R}G)$ with rows $M_1,\ldots,M_n$, the the matrix $M+\alpha_M I_n$ is a sum of squares, where
    \[
    \alpha_M = \max_{i = 1,…,n} ||M_i||_1.
    \]
\end{theorem} 

\begin{proof} 
    First, let us consider a $*$-invariant $x \in \mathbb{R}G$. In this case $x = \sum_{g \in G} c_g g$ and, by our assumption, $c_g = c_{g^{-1}}$. This allows us to write $x$ in the following form
    $$
    x = \sum_{g \in G} \frac{c_g}{2}(g + g^{-1}),
    $$
    and by adding its first norm we obtain
    $$
    x + ||x||_1 = x + \sum_{g \in G} |c_g| = \sum_{g \in G} \frac{|c_g|}{2}( \pm g \pm g^{-1} + 2) = \sum_{g \in G} \frac{|c_g|}{2}(1 \pm g)^*(1 \pm g) \geq 0.
    $$

    Let $E_{i,j}$ be the matrix from $\mathbb{M}_n(\mathbb{R}G)$ with the identity at entry $(i,j)$ and $0$ elsewhere. Then any $\ast$-invariant matrix $M = [m_{i,j}]_{1\leq i,j\leq n} \in \mathbb{M}_n(\mathbb{R}G)$ can be decomposed into a sum of matrices of the two forms.

    \begin{itemize}
        \item For $i < j$:
        $$
        M_{i,j} = m_{i,j} \cdot E_{i,j} + m_{i,j}^* \cdot E_{j,i}, \textrm{ for $m_{i,j} \in \mathbb{R}G$.}
        $$
        \item On the diagonal:
        $$
        M_{i,i} = m_{i,i} \cdot E_{i,i}, \textrm{ for $*$-invariant $m_{i,i} \in \mathbb{R}G$.}
        $$
    \end{itemize}

    By the previous argumentation $M_{i,i} + ||m_{i,i}||_1 \cdot E_{i,i}$ admits a decomposition as a sum of squares. Moreover, by introducing
    $$
    N_{i,j} = M_{i,j} + ||m_{i,j}||_1 \cdot E_{i,i} + ||m_{i,j}||_1 \cdot E_{j,j},
    $$
    and by the formula $E_{i,j} \cdot E_{k,l} = \delta_{j,k} \cdot E_{i,l}$, we obtain
    $$
    M_{i,j} = \frac{1}{2||m_{i,j}||_1}N_{i,j}^*N_{i,j} - \frac{||m_{i,j}||_1^2 + m_{i,j}m_{i,j}^*}{2||m_{i,j}||_1} \cdot E_{i,i} -  \frac{||m_{i,j}||_1^2 + m_{i,j}^*m_{i,j}}{2||m_{i,j}||_1} \cdot E_{j,j}.
    $$
    Now, it is easy to observe that the last two summands are $*$-invariant, and since 
    \[
    ||m_{i,j}^*m_{i,j}||_1, ||m_{i,j}m_{i,j}^*||_1 \leq ||m_{i,j}||_1^2,
    \]
    the first norm of their non-zero entries is not greater than $||m_{i,j}||_1$. Therefore, proceeding as earlier, we add to them the same single-entry matrix multiplied by the norm of the entry, obtaining
    $$
    M_{i,j} + ||m_{i,j}||_1 \cdot E_{i,i} + ||m_{i,j}||_1 \cdot E_{j,j} \geq 0.
    $$
    Hence, in order to obtain a non-negative element from $M$, we need to add $E_{i,i}$ for all possible summands $M_{i,j}$ and $M_{j,i}$ i.e. $E_{i,i}$ scaled by the sum of the first norms $||m_{i,j}||_1$ for $i \leq j$ and $||m_{j,i}||_1$ for $j<i$, which in the case of $*$-invariant matrix is equal to the first norm of the $i$-th row:
    \begin{align*}
         \sum_{i=1}^n \left(M_{i,i} + ||m_{i,i}||_1 \cdot E_{i,i}\right) + \sum_{1\leq i<j \leq n}^n \left(M_{i,j} + ||m_{i,j}||_1 \cdot E_{i,i} + ||m_{i,j}||_1 \cdot E_{j,j}\right) = \\
        = M + \sum_{i=1}^n \left( \left(\sum_{i \leq j} ||m_{i,j}||_1 + \sum_{j < i} ||m_{j,i}||_1 \right) \cdot E_{i,i}\right) = M + \sum_{i = 1}^n ||M_i||_1 \cdot E_{i,i}.
    \end{align*}
    This implies that by adding the identity matrix multiplied by the maximum of the norms between the rows, we obtain a matrix that admits a sum of squares decomposition.

\end{proof}

\section*{Acknowledgments}

Both authors were partially supported by the National Science Center of Poland SONATINA 6 \emph{Algebraic spectral gaps in group cohomology} (grant agreement no.
UMO-2022/44/C/ST1/00037). Additionally, the second author was also partially supported by the National Science Center of Poland MAESTRO 13 \emph{Analysis on Groups} (grant agreement no. UMO-2021/42/A/ST1/00306).



\bibliographystyle{alpha}
\bibliography{sp_2n_z}

\end{document}